\documentclass[11pt,reqno,a4paper]{amsart}

\usepackage{mathrsfs}
\usepackage{amssymb,amsmath,amsfonts,amsthm,mathtools,mathabx,bigints}
\usepackage{placeins}   
\usepackage{mathtools}
\usepackage{calrsfs}
\usepackage{multirow}
\usepackage{bbm}

\newtheorem{theorem}{Theorem}[section]
\newtheorem{lemma}[theorem]{Lemma}
\newtheorem{propos}[theorem]{Proposition}

\theoremstyle{definition}
\newtheorem{defin}[theorem]{Definition}

\numberwithin{equation}{section}

\allowdisplaybreaks

\newcommand{\dive}{{\rm{div}}}
\newcommand{\loc}{\rm{loc}}
\newcommand{\La}{\mathcal{L}}
\newcommand{\p}{\partial} 
\newcommand{\IR}{\mathbb{R}}
\newcommand{\irn}{\IR^N}
\newcommand{\wbar}{\widebar}
\newcommand{\lrn}{L^1(\irn)}
\newcommand{\lprn}{L^p(\irn)}
\newcommand{\IS}{\mathcal{S}}
\newcommand{\U}{\mathcal{U}}
\newcommand{\Z}{\mathcal{Z}}
\newcommand{\Riesz}{\mathcal{R}}
\newcommand{\curr}{{J}}
\newcommand{\irnxv}{\mathbb{R}^N_x \times \mathbb{R}^N_v}
\newcommand{\1}{\mathbbm{1}}
\begin{document}
\title[Lagrangian solutions for Vlasov-Poisson with $L^1$ density]{Lagrangian solutions to the Vlasov-Poisson \\ system with $L^1$ density}

\author{Anna Bohun}
\address{Anna Bohun, Departement Mathematik und Informatik, Universit\"at Basel, 
Rheinsprung 21, CH-4051, Basel, Switzerland}
\email{anna.bohun@unibas.ch}
\author{Fran\c{c}ois Bouchut}
\address{Universit\'e Paris-Est, Laboratoire d'Analyse et de Math\'ematiques Appliqu\'ees (UMR 8050),
CNRS, UPEM, UPEC, F-77454, Marne-la-Vall\'ee, France}
\email{francois.bouchut@u-pem.fr} 
\author{Gianluca Crippa}
\address{Gianluca Crippa, Departement Mathematik und Informatik, Universit\"at Basel, 
Rheinsprung 21, CH-4051, Basel, Switzerland}
\email{gianluca.crippa@unibas.ch}

\begin{abstract} The recently developed theory of Lagrangian flows for transport equations
with low regularity coefficients enables to consider non BV vector fields. We apply this theory to
prove existence and stability of global Lagrangian solutions to the repulsive
Vlasov-Poisson system with only integrable initial distribution function with finite energy.
These solutions have a well-defined Lagrangian flow.
An a priori estimate on the smallness of the superlevels of the flow in three dimensions is established
in order to control the characteristics.
\end{abstract}

\keywords{Vlasov-Poisson system, Lagrangian flows, non BV vector fields, superlevels, weakly convergent initial data}

\maketitle

\section{Introduction}

We consider the Cauchy problem for the classical Vlasov-Poisson system
\begin{equation}\label{vlasov}
\p_t f+v\cdot \nabla_x f +E \cdot \nabla_v f=0,
\end{equation}
\begin{equation}
	f(0,x,v)=f^0(x,v) \,,
	\label{eq:vlasovinit}
\end{equation}
where $f(t,x,v)\geq 0$ is the distribution function, $t\geq 0$, $x,v\in\IR^N$, and
\begin{equation}
	E(t,x)=- \nabla_x U(t,x)
	\label{eq:fieldpot}
\end{equation}
is the force field. The potential $U$ satisfies the Poisson equation
\begin{equation}\label{potential}
-\Delta_x U=  \omega(\rho(t,x)-\rho_b(x)),
\end{equation}
with $\omega=+1$ for the electrostatic (repulsive) case, $\omega=-1$ for the gravitational (attractive) case,
and where the density $\rho$ of particles is defined through
\begin{equation}
\rho(t,x)=  \int_{\IR ^N} f(t,x,v)dv,
	\label{eq:defrho}
\end{equation}
and $\rho_b\geq 0$, $\rho_b\in L^1(\irn)$ is an autonomous background density.
Since we are in the whole space, the relation \eqref{eq:fieldpot} together with the Poisson equation \eqref{potential}
yield the equivalent relation
\begin{equation}\label{field}
E(t,x)=\frac{\omega}{|S^{N-1}|}\frac{x}{|x|^N}*(\rho(t,x)-\rho_b(x)),
\end{equation}
where the convolution is in the space variable.

The Vlasov-Poisson system has been studied for long.
Existence of local in time smooth solutions in dimension $N=3$ has been obtained
in \cite{horst} after the results of \cite{Ba}. Global smooth solutions have been
proved to exist in \cite{pfaff} (and simultaneously in \cite{LP} with a different method),
with improvements on the growth in time in \cite{schaeff,schaeff2}.
These solutions need a sufficiently smooth initial datum $f^0$.
In particular, the following theorem is a classical result due to Pfaffelmoser \cite{pfaff}.
\begin{theorem}
If $N=3$, let $f^0$ be a non-negative $C^1$ function of compact support defined on $\IR^6$.
Then there are a non-negative $f\in C^1(\IR^7)$ and $U\in C^2(\IR^4)$, which tends to zero at infinity
for every fixed $t$, satisfying equations \eqref{vlasov}-\eqref{eq:defrho} with $\rho_b=0$.
For each fixed $t$, the function $f(t,x,v)$ has compact support. The solution is determined uniquely
by the initial datum $f^0$.
\end{theorem}
In a different spirit, global weak solutions were proved to exist in \cite{Ar,DPLVP,DPLkin}, with only
$f^0\in L^1(\IR^6)$, $f^0\log^+f^0\in L^1$, $|v|^2f^0\in L^1$, $E^0\in L^2$ (and $\rho_b=0$, $\omega=+1$).
Related results with {\sl weak} initial data have been obtained in \cite{P,Ja,ZX}.

In this paper we would like mainly to extend the existence result of \cite{DPLVP} to initial data in $L^1$
with finite energy (in the repulsive case $\omega=+1$), avoiding the $L\log^+\!L$ assumption.
Our existence result is Theorem \ref{existence}. It involves a well-defined flow.
Even weaker solutions were considered in \cite{ZM,MMZ,MMZ2},
where the distribution function is a measure. However, these solutions do not have well-defined characteristics.

Our approach uses the theory of Lagrangian flows for transport equations with vector fields having
weak regularity, developed in \cite{lions,Amb,AC,CD,AC2}, and recently in \cite{BC,ACF,BBC}. It enables to consider force fields that
are not in $W^{1,1}_{loc}$, nor in $BV_{loc}$. In this context we prove stability results
with strongly or weakly convergent initial distribution function. The flow is proved to converge strongly anyway.
Our main results were announced in \cite{BCX}. Related results can be found in \cite{ACFVP}.

\section{Conservation of mass and energy}

We would like here to recall some basic identities related to the VP system.
Integrating \eqref{vlasov} with respect to $v$ and noting that the last term is in $v$-divergence form we obtain the {\em local conservation of mass}
\begin{equation}\label{conservationofmass}
\p_t\rho(t,x)+\dive_x(\curr(t,x))=0, 
\end{equation}
where the current $\curr$ is defined by
\begin{equation}\label{current}
\curr(t,x)=\int_{\IR^N}v f(t,x,v) \, dv.
\end{equation}
Integrating again with respect to $x$, we obtain the {global conservation of mass}
\begin{equation}
	\frac{d}{dt}\iint\limits_{\IR^N\times\IR^N}f(t,x,v)dxdv=\frac{d}{dt}\int\limits_{\IR^N}\rho(t,x) dx = 0.
	\label{eq:masscons}
\end{equation}
Multiplying \eqref{vlasov} by $ \frac{|v|^2}{2}$, integrating in $x$ and $v$, we get after integration by parts in $v$
\begin{equation}
	\frac{d}{dt}\iint\limits_{\IR^N\times\IR^N} \frac{|v|^2	}{2}f(t,x,v)dxdv - \iint\limits_{\IR^N\times\IR^N} E \cdot v f\,dxdv=0.
	\label{eq:energyfirst}
\end{equation}
Using \eqref{field} and \eqref{conservationofmass}, one has
\begin{equation}
	\partial_t E+\sum_{k=1}^N\frac{\partial}{\partial x_k}\left(\frac{\omega}{|S^{N-1}|}\frac{x}{|x|^N}*\curr_k\right)=0,
	\label{eq:evolE}
\end{equation}
or in other words
\begin{equation}
	\partial_t E=\omega\nabla_x(-\Delta_x)^{-1}\dive_x \curr,
	\label{eq:dtEdivj}
\end{equation}
which means that $\partial_t E$ is the the gradient component of $-\omega \curr$ (Helmholtz projection).
We deduce that
\begin{equation}
	\int_{\IR^N}E\cdot\partial_t E\,dx=-\omega\int_{\IR^N}E\cdot \curr\,dx.
	\label{eq:intdtE}
\end{equation}
Using \eqref{current} in \eqref{eq:energyfirst}, we obtain the {\em conservation of energy}
\begin{equation}\label{conservationenergy}
\frac{d}{dt}\left[\iint\limits_{\IR^N\times\IR^N} \frac{|v|^2}{2}f(t,x,v)dxdv+\frac{\omega}{2}\int\limits_{\IR^N } |E(t,x)|^2dx \right]=0 \,.
\end{equation}
The total conserved energy is the sum of the kinetic energy and of the potential energy multiplied by the factor $\omega=\pm 1$.
In particular, in the electrostatic case $\omega = +1$ we deduce from \eqref{conservationenergy} a uniform bound in time on both the kinetic and the potential energy,
assuming that they are finite initially.
In the gravitational case $\omega = -1$ it is not possible to exclude that the individual terms of the kinetic and potential energy become unbounded in finite time,
while the sum remains constant. Indeed it is known that it does not happen in three dimensions as soon as
$f^0$ is sufficiently integrable, but we cannot exclude this a priori for only $L^1$ solutions.\\

Note that the assumption $E^0\in L^2$ is satisfied in 3 dimensions as soon as
$\rho^0-\rho_b\in L^{6/5}$. However, in one or two dimensions, for $E^0$ to be in $L^2$ it is necessary
that $\int(\rho^0-\rho_b)dx=0$, as is easily seen in Fourier variable.
It is also necessary that $\rho^0-\rho_b$ has enough decay at infinity.
Thus in one or two dimensions, in order to have finite energy, $\rho_b$ cannot be zero identically.

\section{Regularity of the force field for $L^1$ densities}

\subsection{Singular integrals}\label{singularint}

\begin{defin}\label{skft}
A function $K$ is a \emph{singular kernel of fundamental type} in $\IR^{N}$ if the following properties hold:
\begin{enumerate}
\item $K|_{\IR^N\setminus\{0\}}\in C^1(\IR^N\setminus\{0\})$.
\item There exists a constant $C_0\geq 0$ such that
\begin{equation}\label{c0}
\begin{aligned}
\begin{array}{ccc}
|K(x)|\leq \frac{C_0}{|x|^N}, && x\in \irn\setminus\{0\}.
\end{array}
\end{aligned}
\end{equation}
\item There exists a constant $C_1\geq 0$ such that
\begin{equation}\label{c1}
\begin{aligned}
\begin{array}{ccc}
|\nabla K(x)|\leq \frac{C_1}{|x|^{N+1}},
&& x\in \irn\setminus\{0\}.
\end{array}
\end{aligned}
\end{equation}
\item There exists a constant $A_2\geq 0$ such that 
\begin{equation}
\left|\int_{R_1<|x|<R_2}K(x) dx\right| \leq A_2,
\end{equation}
for every $0<R_1<R_2<\infty$.
\end{enumerate}
\end{defin}

\begin{theorem}[Calder\'on-Zygmund]\label{skthm}

A singular kernel of fundamental type $K$ has an extension as a distribution on $\IR^N$
(still denoted by $K$), unique up to a constant times Dirac delta at the origin, such that
$\hat K\in L^\infty(\IR^N)$. Define
\begin{equation}
	Su=K* u, \quad \textrm{for } u\in L^2(\irn),
\end{equation}
in the sense of multiplication in the Fourier variable. Then we have the estimates for $1<p<\infty$
\begin{equation}\label{czk}
\begin{array}{cc}
||Su||_{L^p(\irn)} \leq C_{N,p}(C_0+C_1+||\hat{K}||_{L^\infty})||u||_{L^p(\irn)}, & u\in L^p\cap L^2(\irn).
\end{array}
\end{equation}
\end{theorem} 
If $K$ is a singular kernel of fundamental type, we call the associated operator $S$ a singular integral operator on $\irn$.
We define then the Fr\'echet space $\Riesz(\IR^N)=\cap_{m\in {\mathbb{N}},\,1<p<\infty}W^{m,p}(\IR^N)$
and its dual $\Riesz'(\IR^N)\subset \IS'(\irn)$, where $\IS'(\irn)$ is the space of tempered distributions on $\IR^N$.
Since all singular integral operators are bounded on $\Riesz(\IR^N)$, by duality we can define the operator
$S$ also $\Riesz'(\IR^N)\to \Riesz'(\IR^N)$.
In particular it enables to define $Su$ for $u\in L^1(\IR^N)$ or for $u$ a measure.
The result $Su$ is in $\Riesz'(\IR^N)\subset \IS'(\irn)$.

\subsection{The split vector field}

Let $\rho(t,x)\in L^\infty((0,T); L^1(\IR^N))$. We denote by 
\begin{equation}\label{bvf}
b(t,x,v)=(b_1,b_2)(t,x,v)=(v,E(t,x))=\Bigl(v,-\omega\nabla_x(-\Delta_x)^{-1}(\rho(t,x)-\rho_b(x))\Bigr)
\end{equation}
the associated vector field on $(0,T) \times \IR^N\times\IR^N$.
Then the Vlasov equation can be written in the form of the transport equation $\p_t f+b\cdot\nabla_{x,v} f=0$. 
In the following subsections we establish bounds on the vector field $b$.

\subsection{Local integrability}\label{localintegr}

For $L^1$ densities, we have the weak estimates from the Hardy-Littlewood-Sobolev inequality:
\begin{equation}\begin{array}{l}
	\displaystyle\hphantom{=\ } \left|\left|\left|\nabla (-\Delta)^{-1}(\rho(t,x)-\rho_b(x))\right|\right|\right|_{M^{\frac{N}{N-1}}(\IR^N)}\\
\displaystyle \leq \left|\left|\left|{\frac{1}{|S^{N-1}|}|x|^{1-N}}*|\rho(t,x)-\rho_b(x)| \right|\right|\right|_{M^{\frac{N}{N-1}}(\IR^N)}\\
\displaystyle\leq c_N ||\rho(t,x)-\rho_b(x)||_{L^1(\IR^N)},
	\label{kern}
	\end{array}
\end{equation}
where $|||u|||_{M^p(\IR^N)}\equiv\sup_{\gamma>0}\gamma\mathop{\La^N}(\{x\in\IR^N\,s.t.\, |u(x)|>\gamma\})^{1/p}$.
It follows that
\begin{equation}
|||E |||_{L^\infty((0,T);M^{\frac{N}{N-1}}( \IR^N))} \leq c_N||\rho-\rho_b||_{ L^\infty((0,T);L^{1}( \IR^N))},
\end{equation}
and using the inclusion $M^{\frac{N}{N-1}} (\irn) \subset L^p_{\loc}(\irn)$ for $1\leq p<\frac{N}{N-1}$ we conclude that
$b\in L^\infty((0,T);L^p_{\loc}(\irnxv))$ for any $1\leq p<\frac{N}{N-1}$, since $v\in L^p_{\loc}(\irnxv)$ for any $p$.

\subsection{Spatial regularity}\label{spatialreg}

Since $b_1=v$ is smooth, the only non-trivial gradient is the one of $b_2=E$,
indeed the differential matrix of the vector field is given by
\begin{equation}
Db =
\left( \begin{matrix}
D_x b_1 & D_{v}b_1\\
D_{x}b_2 & D_{v}b_2 
\end{matrix} \right)
=\left( \begin{matrix} 
 0 & {\rm Id} \\
D_{x} E & 0
\end{matrix} \right).
\end{equation}
We have  by \eqref{bvf}
\begin{equation}
(D_{x} E)_{ij}\equiv\partial_{x_j}E_i = -\omega\partial^2_{x_ix_j}((-\Delta_x)^{-1}(\rho-\rho_b))
\quad\mbox{for }1\leq i,j \leq N.
\end{equation}
It is well-known that the operator $\partial^2_{x_ix_j}(-\Delta_x)^{-1}$ is a singular integral operator.
Its kernel is
\begin{equation}
	K_{ij}(x)=-\frac{1}{|S^{N-1}|}\frac{\partial}{\partial_{x_j}}\left(\frac{x_i}{|x|^N}\right),
	\label{eq:kernelkij}
\end{equation}
it is given outside of the origin by 
\begin{align} \label{n3}
K_{ij}(x)=\frac{1}{|S^{N-1}|}\left(N\frac{x_ix_j}{|x|^{N+2}}-\frac{\delta_{ij}}{|x|^N}\right),
	\quad\mbox{ for }x\in\IR^N\backslash\{0\}.
\end{align}
The kernel satisfies the conditions of Subsection \ref{singularint},
and $\hat K_{ij}(\xi)=-\xi_i\xi_j/|\xi|^2$.
Thus (each component of) $D_xE$ is a singular integral of an $L^\infty((0,T);L^1(\IR^N))$ function.

\subsection{Time regularity}

According to \eqref{eq:dtEdivj}, $\partial_t E$ is a singular integral of the current $\curr$
defined  by \eqref{current}.
Using the bounds available for solutions with finite mass and energy
\begin{equation}
	||f (t,\cdot)||_{L^1(\IR^N_x\times\IR^N_v)},  \iint|v|^2f (t,x,v)dxdv  \leq C,
	\label{eq:boundmen}
\end{equation}
and since $|v|\leq 1+|v|^2$, we get that $\curr \in L^\infty((0,T);L^1(\IR^N_x))$.
Hence $\partial_t E$ is a singular integral of an $L^\infty((0,T);L^1(\IR^N_x))$ function.
In particular,
\begin{equation}
\p_t E\in L^\infty((0,T);\IS'(\irn)).
\end{equation}

\section{Lagrangian flows}

Suppose that $f$ is a smooth solution to \eqref{vlasov}-\eqref{eq:defrho} with $f(0,x,v)=f^0$. Then $f$ is constant along the characteristics $(X(s,t,x,v),V(s,t,x,v))$, which solve the system of equations
\begin{equation}
	\begin{cases}
\displaystyle \frac{dX}{ds}(s,t,x,v)=V(s,t,x,v), \\ \\
\displaystyle \frac{dV}{ds}(s,t,x,v)=E (s,X(s,t,x,v),V(s,t,x,v)),
\end{cases}
	\label{eq:chareq}
\end{equation}
with initial data $X(t,t,x,v)=x$ and $V(t,t,x,v)=v$. 
Thus the solution can be expressed as $f(t,x,v)=f^0(X(0,t,x,v),V(0,t,x,v))$. 

In order to extend this notion of characteristics to non-smooth solutions, we define regular Lagrangian flows, which are defined in an almost everywhere sense.

\begin{defin}\label{flow}
Let $b\in L^1_{\rm{loc}}([0,T]\times\IR^{2N};\IR^{2N})$, and $t\in[0,T)$.
A map $Z:[t,T]\times\IR^{2N}\rightarrow\IR^{2N}$ is a regular Lagrangian flow
starting at time $t$ for the vector field $b$ if
\begin{itemize}
\item[(1)] for a.e. $z\in\IR^{2N}$ the map $s\mapsto Z(s,z)$ is an absolutely continuous integral solution of $\dot{\beta}(s)=b(s,\beta(s))$ for $s\in[t,T]$ with $\beta(t)=z$.
\item[(2)] There exists a constant $L$ independent of $s$ such that
$$\La^{2N}(Z(s,\cdot)^{-1}(A))\leq L\La^{2N}(A)$$ for every Borel set $A\subset\IR^{2N}$.
\end{itemize}
The constant $L$ is called the compressibility constant of $Z$. 
\end{defin}

\begin{defin}\label{sublevel}
Define the sublevel of the flow as the set
\begin{equation}
	G_\lambda=\{z\in\IR^{2N}: |Z(s,z)|\leq\lambda \text{ for almost all } s\in[t,T]\}.
	\label{eq:sublevel}
\end{equation}
\end{defin}
 
\section{Stability estimate for Lagrangian flows}

We summarize the main result from \cite{BC} in the following regularity setting of the vector field in arbitrary dimension.
We say that a vector field $b$ satisfies {\bf(R1)} if $b$ can be decomposed as
\begin{equation}
\frac{b(t,z)}{1+|z|} = \tilde{b}_1(t,z)+\tilde{b}_2(t,z) 
\end{equation}
where $\tilde{b}_1\in L^1((0,T);L^1(\IR^{2N}))$, $\tilde{b}_2\in L^1((0,T);L^\infty(\IR^{2N}))$.

We assume also that $b$ satisfies {\bf (R2)}: for every $j=1,\ldots,2N$, 
\begin{equation}\label{r2}
\p_{z_j} b = \displaystyle\sum_{k=1}^m S_{jk} g_{jk}
\end{equation}
where $S_{jk}$ are singular integrals of fundamental type on $\IR^{2N}$ and $g_{jk} \in L^1((0,T);L^1(\IR^{2N}))$.
Moreover, we assume condition {\bf (R3)}, that is
\begin{equation}\label{R3}
b \in L^p_{\loc}([0,T]\times\IR^{2N}), \qquad \mbox{ for some } p>1.
\end{equation}
\\
We recall the following stability theorem from \cite{BC}, where we denote by $B_r$
the ball with center $0$ and radius $r$ in $\IR^{2N}$.
\begin{theorem}\label{fund}
Let $b$, $\bar{b}$ be two vector fields satisfying {\bf (R1)}, $b$ satisfying also {\bf (R2), (R3)}. Fix $t\in[0,T)$ and let $Z$ and $\bar{Z}$ be regular Lagrangian flows starting at time $t$ associated to $b$ and $\bar{b}$ respectively, with compression constants $L$ and $\bar{L}$.
Then the following holds.
\\For every $\gamma>0$, $r>0$ and $\eta>0$ there exist $\lambda>0$  and $C_{\gamma,r,\eta}>0$ such that
$$\La^{2N}\left(B_r\cap \{|Z(s,\cdot)-\bar{Z}(s,\cdot)|>\gamma\}\right)\leq C_{\gamma,r,\eta} ||b-\bar{b}||_{L^1((0,T)\times B_\lambda)}+\eta$$ for all $s\in[t,T]$. 
The constants $\lambda$ and $C_{\gamma,r,\eta}$ also depend on 
\begin{itemize}
 \item The equi-integrability in $L^1((0,T);L^1(\IR^{2N}))$ of $g_{jk}$ coming from {\bf(R2)},
\item The norms of the singular integral operators $S_{jk}$ from {\bf(R2)},
\item The norm $||b||_{L^p((0,T)\times B_\lambda)}$ corresponding to {\bf(R3)},
\item The $L^1(L^1)$ and $L^1(L^\infty)$ norms of the decompositions of $b$ and $\bar{b}$ in {\bf(R1)},
\item The compression constants $L$ and $\bar{L}$. 
\end{itemize}
\end{theorem}
We would like now to state a variant of this theorem, where {\bf(R1)} and {\bf(R2)}
are replaced by {\bf(R1a)} and {\bf(R2a)}.\\

We consider the following weakened assumption {\bf (R1a)}: 
for all regular Lagrangian flow $Z:[t,T]\times\IR^{2N}\rightarrow\IR^{2N}$ relative to $b$
starting at time $t$ with compression constant $L$, and for all $r,\lambda>0$,
\begin{equation}\label{r1a}
 \La^{2N}(B_r\setminus G_\lambda)\leq g(r,\lambda),
	\quad\mbox{with }g(r,\lambda)\to 0 \text{ as } \lambda\to\infty\mbox{ at fixed }r,
\end{equation}
where $G_\lambda$ denotes the sublevel of the flow $Z$, defined in \eqref{eq:sublevel}.\\

We next consider a splitting of the variables, $\IR^{2N}_z=\IR^N_x\times\IR^N_v$,
and we denote $Z(t,x,v)=(X,V)(t,x,v)$.
Noticing the special form \eqref{bvf}, we assume the condition that $b$ satisfies {\bf (R2a)}: 
\begin{equation}
b(t,x,v)=(b_1,b_2)(t,x,v)=(b_1(v),b_2(t,x)),
\end{equation}
with
\begin{equation}
	b_1 \in \textrm{Lip} (\IR^{N}_v),
	\label{eq:hyplipb1}
\end{equation}
and where $b_2$ involves singular kernels only in the first set of variables $x$, that is for every $j=1,...,N$,
\begin{equation}
\p_{x_j} b_2 = \displaystyle\sum_{k=1}^m S_{jk} g_{jk}, 
\end{equation}
where $S_{jk}$ are singular integrals of fundamental type on $\IR^{N}$ and $g_{jk} \in L^1((0,T);L^1(\IR^{N}))$.
\begin{theorem}\label{Th estimatestab}
Let $b$, $\bar{b}$ be two vector fields satisfying {\bf(R1a)}, $b$ satisfying also {\bf (R2a), (R3)}.
Fix $t\in[0,T)$ and let $Z$ and $\bar{Z}$ be regular Lagrangian flows starting at time $t$
associated to $b$ and $\bar{b}$ respectively, with compression constants $L$ and $\bar{L}$,
and sublevels $G_\lambda$ and $\bar{G}_\lambda$.
Then the following holds.
\\For every $\gamma>0$, $r>0$ and $\eta>0$, there exist $\lambda>0$  and $C_{\gamma,r,\eta}>0$ such that
$$\La^{2N}\left(B_r\cap \{|Z(s,\cdot)-\bar{Z}(s,\cdot)|>\gamma\}\right)\leq C_{\gamma,r,\eta} ||b-\bar{b}||_{L^1((0,T)\times B_\lambda)}+\eta$$
for all  $s\in[t,T]$. 
The constants $\lambda$ and $C_{\gamma,r,\eta}$ also depend on 
\begin{itemize}
\item The equi-integrability in $L^1((0,T);L^1(\IR^{N}))$ of $g_{jk}$ coming from {\bf(R2a)},
\item The norms of the singular integral operators $S_{jk}$ from {\bf(R2a)}, 
\item The Lipschitz constant of $b_1$ from {\bf(R2a)}, 
\item The norm $||b||_{L^p((0,T)\times B_\lambda)}$ corresponding to {\bf(R3)},
\item The rate of decay of $\La^{2N}(B_r\setminus  {G}_\lambda) $ and $\La^{2N}(B_r\setminus \bar{G}_\lambda)$ from {\bf(R1a)},
\item The compression constants $L$ and $\bar{L}$. 
\end{itemize}
\end{theorem}
\begin{proof} We summarize the modifications of the proof of Theorem \ref{fund} from \cite{BC}.
The main added difficulty is the singular integral operators on $\IR^{N}_x$ instead of $\IR^{2N}$.
More general situations with singular integrals of measures instead of $L^1$ functions
are considered in \cite{BBC}.
Let $z=(x,v)\in \IR^{N}\times\IR^{N}$. We estimate the quantity
\begin{equation}
	\Phi_\delta(s) =  \int\limits_{B_r\cap G_\lambda \cap \wbar{G}_\lambda} \log\left(1+\frac{|Z(s,z)-\wbar{Z}(s,z)|}{\delta}\right) dz \,.
	\label{eq:defPhidelta}
\end{equation}
Differentiating with respect to time, we get
\begin{align*}
 \Phi_{\delta}'(s)\leq &\int\limits_{B_r\cap G_\lambda \cap \wbar{G}_\lambda} \frac{|b(s,Z(s,z))-\wbar{b}(s,\wbar{Z}(s,z))|}{\delta+|Z(s,z)-\wbar{Z}(s,z)|} dz\\
 \leq &\frac{\wbar{L}}{\delta}||b(s,\cdot)-\wbar{b}(s,\cdot)||_{L^1(B_\lambda)} \\
&+\int\limits_{B_r\cap G_\lambda \cap \wbar{G}_\lambda} \min\biggl\{\frac{|b(s,Z(s,z))|+|b(s,\wbar{Z}(s,z))|}{\delta}, \\
& \frac{|b_1(s,Z(s,z))-b_1(s,\wbar{Z}(s,z))|}{|Z(s,z)-\wbar{Z}(s,z)|}
+ \frac{|b_2(s,Z(s,z))-b_2(s,\wbar{Z}(s,z))|}{|Z(s,z)-\wbar{Z}(s,z)|} \biggr\}dz.
\end{align*}
Using the special form of $b$ from {\bf(R2a)}, we obtain 
\begin{equation}
\begin{aligned}
 \Phi_{\delta}'(s)\leq &\frac{\wbar{L}}{\delta}||b(s,\cdot)-\wbar{b}(s,\cdot)||_{L^1(B_\lambda)} \\
&+\int\limits_{B_r\cap G_\lambda \cap \wbar{G}_\lambda} \min\biggl\{\frac{|b(s,Z(s,z))|+|b(s,\wbar{Z}(s,z))|}{\delta},\\
&\quad \frac{|b_1(V(s,z))-b_1(\wbar V(s,z))|}{|V(s,z)-\wbar V(s,z)|}
  + \frac{|b_2(s,X(s,z))-b_2(s,\wbar X(s,z))|}{|X(s,z)-\wbar X(s,z)|} \biggr\}dz \\ 
    \leq &\frac{\wbar{L}}{\delta}||b(s,\cdot)-\wbar{b}(s,\cdot)||_{L^1(B_\lambda)} + \textrm{Lip}(b_1)\La^{2N}(B_r) \\
  &+  \int\limits_{B_r\cap G_\lambda \cap \wbar{G}_\lambda} \min\biggl\{\frac{|b(s,Z(s,z))|+|b(s,\wbar{Z}(s,z))|}{\delta},\\
   &  \mkern 150mu\frac{|b_2(s,X(s,z))-b_2(s,\wbar X(s,z))|}{|X(s,z)-\wbar X(s,z)|} \biggr\}dz.
  \end{aligned} 
\end{equation}
Using assumption {\bf(R2a)}, we can use the estimate of \cite{BC} on the difference quotients of $b_2$,
\begin{equation}
	\frac{|b_2(s,X(s,z))-b_2(s,\wbar X(s,z))|}{|X(s,z)-\wbar X(s,z)|}
	\leq \U(s,X(s,z))+\U(s,\wbar X(s,z)),
	\label{eq:diffquot}
\end{equation}
where $\U(s,.)\in M^1(\IR^N)$ for fixed $s$ is indeed given by
\begin{equation}
	\U(s,\cdot)=\sum_{j=1}^N\sum_{k=1}^mM_j(S_{jk}g_{jk}(s,.)),
	\label{eq:defUmax}
\end{equation}
with $M_j$ a smooth maximal operator on $\IR^N$.
Next, we can define the function
\begin{equation}
	\Z(s,x,v)=\U(s,x)\1_{(x,v)\in \wbar{B_{\lambda}}},
	\label{eq:defZmax}
\end{equation}
and we notice that since the above integrals are over $G_\lambda \cap \wbar{G}_\lambda$, we can
replace the right-hand side of \eqref{eq:diffquot} by $\Z(s,Z(s,z))+\Z(s,\wbar Z(s,z))$.
Then, for given $\varepsilon>0$, we can decompose
$g_{jk}=g_{jk}^1+g_{jk}^2$, with $\|g_{jk}^1\|_{L^1((0,T)\times\IR^N)}\leq \varepsilon$ and
$\|g_{jk}^2\|_{L^2((0,T)\times\IR^N)}\leq C_\varepsilon$, $g_{jk}^2$ having support in a
set $A_\varepsilon$ of finite measure.
This gives rise to two type of terms $\Z^1$ and $\Z^2$.
Since all the $\Z$ terms have compact support in $v$, this allows to perform the same
estimates as in \cite{BC}. We finally get
\begin{equation}
     \begin{aligned}
& \La^{2N}(B_r\cap\{|Z(\tau,\cdot)-\wbar{Z}(\tau,\cdot)|>\gamma\}) \\
 \leq & \frac{1}{\log(1+\gamma/\delta)}\int_t^\tau\Phi_\delta'(s)\,ds
 +\La^{2N}(B_r\setminus G_\lambda)+ \La^{2N}(B_r\setminus \wbar{G}_\lambda) .
     \end{aligned} 
\end{equation}
By choosing $\lambda$ large first, then $\varepsilon$ small, and finally $\delta$ small,
we conclude the proof as in \cite{BC}.
\end{proof}

\section{Control of superlevels}

In order to apply Theorem \ref{Th estimatestab}, we need to satisfy {\bf(R1a)}.
Therefore, we seek an upper bound on the size of $B_r\setminus G_\lambda.$

\subsection{The case of low space dimension} 

We first recall the following lemma from \cite{BC}.
\begin{lemma}
Let $b:(0,T)\times\IR^{2N}\rightarrow \IR^{2N}$ be a vector field satisfying  {\bf (R1)}.
Then $b$ satisfies {\bf(R1a)}, where the function $g(r,\lambda)$ depends only on $L,||\Tilde{b}_1||_{L^1((0,T);L^1(\IR^{2N}))},||\Tilde{b}_2||_{L^1((0,T);L^\infty(\IR^{2N}))}$.
\end{lemma}
This lemma allows us to control the superlevels of $b$ in $1$ or $2$ dimensions.
\begin{propos}\label{estsuper2}
Let $b$ be the vector field in \eqref{bvf}, with $E\in L^\infty((0,T);L^2(\IR^N))$.
For $N=2$ or $N=1$, $b$ satisfies {\bf (R1)}, hence also {\bf(R1a)}.
\end{propos}
\begin{proof} It is clear that
\begin{equation}
	\frac{v}{1+|x|+|v|}\in L^\infty_t(L^\infty_{x,v}),
\end{equation}
and
\begin{equation}
\frac{E(t,x)}{1+|x|+|v|}=\frac{E(t,x)}{1+|x|+|v|}\1_{|v|\leq|E(t,x)|} +\frac{E(t,x)}{1+|x|+|v|}\1_{|v|>|E(t,x)|}\equiv \tilde{E_1}+\tilde{E_2}.
\end{equation}
Clearly $\tilde{E_2}\in L^\infty_t(L^\infty_{x,v})$, and if $N=2$, $\tilde{E_1}\in L^\infty((0,T);L^1_{x,v})$ since
\begin{equation}\begin{array}{l}
	\displaystyle \iint\limits_{\IR^2\times\IR^2}\frac{|E(t,x)|}{1+|x|+|v|}\1_{|v|\leq|E(t,x)|}dxdv\\
	\displaystyle\leq \int\limits_{\IR^2} |E(t,x)|\biggl(\int\limits_{|v|\leq |E(t,x)|}\frac{1}{|v|}dv\biggr)dx
	= 2\pi\int\limits_{\IR^2} |E(t,x)|^2dx.
	\end{array}
\end{equation}
In the case $N=1$, we have directly that $E(t,x)/(1+|v|)\in L^\infty((0,T);L^2_{x,v})$.
\end{proof}

\subsection{The case of three space dimensions}

The condition {\bf (R1)} being not satisfied in $3$ dimensions,
we need an estimate on $|Z|$ in order to control the superlevels.
For getting this we integrate in space a function growing slower at infinity than $\log(1+|Z|)$ (this corresponding to the case {\bf(R1)}).
\begin{propos}\label{estsuper3}
Let $b$ be as in \eqref{bvf} with $N=3$, $E\in L^\infty((0,T);L^2(\IR^N))$, satisfying \eqref{eq:dtEdivj}
with $\curr\in L^\infty((0,T);L^1(\IR^N))$. Furthermore, assume that $\omega=+1$, $\rho\geq 0$, and $\rho_b\in L^1\cap L^p(\IR^3)$ for some $p>3/2$.
Then {\bf(R1a)} holds, where the function $g$ depends only
on $L$, $T$, $||E||_{L^\infty_t(L^2_x)}$, $||\curr||_{L^\infty_t(L^1_x)}$, $\|(-\Delta)^{-1}\rho_b\|_{L^\infty}$,
and one has $g(r,\lambda)\to 0$ as $\lambda\to \infty$ at fixed $r$. 
\end{propos}
\begin{proof}
Step 1.)  Let $Z:[t,T]\times\IR^3\times\IR^3\rightarrow\IR^3\times\IR^3$ be a regular Lagrangian flow relative to $b$
starting at time $t$, with compression constant $L$ and sublevel $G_\lambda$.
Denoting $Z=(X,V)$, we have the ODEs
\begin{align}
\begin{cases}
&\dot{X}(s,x,v)=V(s,x,v), \\
&\dot{V}(s,x,v)=E(s,X(s,x,v)).
\end{cases}
\label{eq:edos}
\end{align}
Recalling that $E=-\nabla_x U$, one has
\begin{align}\label{energydecomp}
&\p_s\frac{|V(s,x,v)|^2}{2}=V(s,x,v)\cdot\p_s V(s,x,v) = E(s,X(s,x,v))\cdot\p_s X(s,x,v) \nonumber \\
&= -\p_s[U(s,X(s,x,v))]+\p_t U(s,X(s,x,v)).
\end{align}
This computation is indeed related to the form of the Hamiltonian for \eqref{vlasov}, ${\mathcal H}=|v|^2/2+U(t,x)$.

We are going to bound the superlevels of $V(s,x,v)$. We claim that
\begin{equation}
	\iint_{B_r} \sup_{s\in[t, T]}\left(1+\log\Bigl(1+|V(s,x,v)|^2/2\Bigr)\right)^\alpha dxdv\leq  A,
\end{equation}
where $0<\alpha<1/3$, and for some constant $A$ depending on $L$, $T$, $r$, $\alpha$, and on the norms
$||E||_{L^\infty_t(L^2_x)}$, $||\curr||_{L^\infty_t(L^1_x)}$, $\|(-\Delta)^{-1}\rho_b\|_{L^\infty}$.
Assume for the moment that this holds. From the lower bound
\begin{equation}\begin{array}{l}
	\displaystyle\hphantom{\geq\ } \iint_{B_r} \sup_{s\in[t, T]}\left(1+\log\Bigl(1+|V(s,x,v)|^2/2\Bigr)\right)^\alpha dxdv\\
	\displaystyle\geq \La^6(B_r\setminus \widetilde G_\lambda) (1+\log(1+\lambda^2/2))^\alpha,
	\end{array}
\end{equation}
with $\widetilde G_\lambda$ the sublevel of $V$, we get that  
\begin{equation}
 \La^6 (B_r\setminus \widetilde G_\lambda)\leq \frac{A}{(1+\log(1+\lambda^2/2))^\alpha}.
\end{equation}
Next, we remark that by the first equation in \eqref{eq:edos}, whenever $(x,v)\in \widetilde G_\lambda$ one has
$|X(s,x,v)|\leq |x|+|s-t|\lambda$, and $|Z(s,x,v)|\leq |x|+(1+T)\lambda$.
Thus for $\lambda>r$, one has $B_r\setminus G_\lambda\subset B_r\setminus \widetilde G_{(\lambda-r)/(1+T)}$,
which enables to conclude the proposition (for $\lambda\leq r$ we can just bound
$\La^6(B_r\setminus G_\lambda)$ by $\La^6(B_r)$).
\\
Step 2.) By Step 1, it is enough to prove that we have a decomposition
 \begin{equation}\label{claim}
 \left(1+\log\left(1+\frac{ |V(s,\cdot,\cdot)| ^2}{2}\right)\right)^\alpha\leq f_1+f_2\in L^1 (\IR^3_x\times\IR^3_v)+L^\infty (\IR^3_x\times\IR^3_v),
 \end{equation}
for $(x,v)\in B_r$, where $f_1$, $f_2$ are independent of $s\in[t,T]$.
Let
\begin{equation}
	\beta(y)= \left(1+\log\left(1+y\right)\right)^\alpha,\quad\mbox{for }y\geq 0.
	\label{eq:defbeta}
\end{equation}
Then
\begin{equation}\begin{array}{c}
	\displaystyle \beta'(y)=\frac{\alpha(1+\log(1+y))^{\alpha-1}}{1+y},\\
	\displaystyle 0<-\beta''(y)\leq \frac{(1+\log(1+y))^{\alpha-1}}{(1+y)^2}.
	\label{eq:boundsbetaprime}
	\end{array}
\end{equation}
Using \eqref{energydecomp}, we compute
 \begin{equation}
 \begin{aligned}
 & \p_s\left[\beta\left(\frac{|V(s,x,v)|^2}{2}\right)\right] \\
  = &\Bigl(-\p_s[U(s,X(s,x,v))]+\p_t U(s,X(s,x,v)) \Bigr)\beta'\left(\frac{|V(s,x,v)|^2}{2}\right)\\
  = & -\p_s\left[U(s,X(s,x,v))\beta'\left(\frac{|V(s,x,v)|^2}{2}\right)\right] \\\
&+ U(s,X(s,x,v))\beta''\left(\frac{|V(s,x,v)|^2}{2}\right)V(s,x,v)\cdot E(s,X(s,x,v))\\
& + \p_t U(s,X(s,x,v))  \beta'\left(\frac{|V(s,x,v)|^2}{2}\right).
 \end{aligned}
  \end{equation}
Thus, integrating between $t$ and $s$,
\begin{equation}\label{int}
\begin{aligned}
 & \left(1+\log\left(1+\frac{ |V(s,x,v)| ^2}{2}\right)\right)^\alpha \\
= & - \frac{\alpha U(s,X(s,x,v))}{\left(1+\frac{|V(s,x,v)|^2}{2}\right)\left(1+\log\left(1+\frac{|V(s,x,v)|^2}{2}\right)\right)^{1-\alpha}} \\
 & + \frac{\alpha U(t,x)}{\left(1+\frac{|v|^2}{2}\right)\left(1+\log\left(1+\frac{|v|^2}{2}\right)\right)^{1-\alpha}} +  \left(1+\log\left(1+\frac{ {|v|^2}}{2}\right)\right)^{\alpha} \\
 & + \int^s_t \left\{ U(\tau,X(\tau,x,v))V(\tau,x,v)\cdot E(\tau,X(\tau,x,v))   \beta''\left(\frac{|V(\tau,x,v)|^2}{2}\right)
    \right. \\
 & \left.  + \p_t U(\tau,X(\tau,x,v)) \beta'\left(\frac{|V(\tau,x,v)|^2}{2}\right)   \right\}d\tau.
\end{aligned}
\end{equation}
Step 3.)
Since $E(t,\cdot)\in L^2(\IR^3)$, we have by the Sobolev embedding that 
$U(t,\cdot)  \in L^6(\IR^3)$. Thus clearly
\begin{equation}
\frac{U(t,x)}{1+\frac{|v|^2}{2}} \in  L^6(\IR^3_x\times\IR^3_v)\subset  L^1 (\IR^3_x\times\IR^3_v)+L^\infty (\IR^3_x\times\IR^3_v).
\end{equation}
Next, since $\omega=+1$ and $\rho\geq 0$, one has $U=U_{\rho}-U_{\rho_b}$, with $U_{\rho}\geq 0$.
Thus $U\geq -\|U_{\rho_b}\|_{L^\infty}$.
Thus the first three terms in the expansion \eqref{int} are upper bounded in $L^1 (\IR^3_x\times\IR^3_v)+L^\infty (\IR^3_x\times\IR^3_v)$.
It remains to estimate the integral. We can bound it by $\Phi_1+\Phi_2$, with
\begin{equation}\label{iint1}
  \Phi_1:= \int^T_t  \left| U(\tau,X(\tau,x,v))V(\tau,x,v)\cdot E(\tau,X(\tau,x,v))
	\beta''\left(\frac{|V(\tau,x,v)|^2}{2}\right)\right| d\tau,
\end{equation}
\begin{equation}\label{iint2}
  \Phi_2:= \int^T_t    \left|\p_t U(\tau,X(\tau,x,v)) \beta'\left(\frac{|V(\tau,x,v)|^2}{2}\right)\right| d\tau.
\end{equation}
Note that $\Phi_1$, $\Phi_2$ are independent of $s$.
We estimate $\Phi_1$ in $L^{3/2}(\IR^3_x\times\IR^3_v)\subset L^1 (\IR^3_x\times\IR^3_v)+L^\infty (\IR^3_x\times\IR^3_v)$.
Passing the $L^{3/2}$ norm under the integral and
changing $(X(\tau,x,v),V(\tau,x,v))$ to $(x,v)$, this gives (up to a factor $L$)
\begin{align*}
&  \int\limits_{\IR^3}\int\limits_{\IR^3} \left|  \frac{U(\tau,x) v \cdot E(\tau,x)    }{\left(1+\frac{|v|^2}{2}\right)^2\left(1+\log\left(1+\frac{|v|^2}{2}\right) \right)^{1-\alpha}  }\right|^{3/2}  dxdv \\
& \leq  
 \int\limits_{\IR^3}|U(\tau,x)   E(\tau,x)|^{3/2} dx     \int\limits_{\IR^3}  \frac{2^{3/4}\,dv }{\left(1+\frac{|v|^2}{2}\right)^{9/4}\left(1+\log\left(1+\frac{|v|^2}{2}\right) \right)^{3(1-\alpha)/2}  } \\
 & \leq c ||U(\tau,\cdot)||_{L^6(\IR^3)}^{3/2}||E(\tau,\cdot)||_{L^2(\IR^3)}^{3/2}.
\end{align*}
Thus $\Phi_1\in  L^{3/2}(\IR^3_x\times\IR^3_v)$. 
\\
Step 4.) For $\Phi_2$, we notice that $E$ satisfies \eqref{eq:dtEdivj} and $E=-\nabla_x U$, thus
\begin{equation}
  \p_t U = -\omega(-\Delta_x)^{-1}\dive_x\curr.
\end{equation}
Since $\curr(\tau,\cdot)\in L^1(\IR^3_x)$, we deduce by the Hardy Littlewood Sobolev inequality that
\begin{equation}
|||\p_t U(\tau,\cdot)|||_{M^{3/2}(\IR^3)}\leq c||\curr(\tau,\cdot)||_{L^1(\IR^3)}.
\end{equation}
Therefore, we estimate
\begin{equation}\label{phi2}
\begin{aligned}
& \left|\left|\left|\frac{\p_t U(\tau,X(\tau,x,v))}{\left(1+\frac{|V(\tau,x,v)|^2}{2}\right)\left(1+\log\left(1+\frac{|V(\tau,x,v)|^2}{2}\right)\right)^{1-\alpha}}\right|\right|\right|_{M^{3/2}(\IR^3_x\times\IR^3_v)} \\
& \leq       L^{2/3}\left|\left|\left| \frac{\p_t U(\tau,x)}{\left(1+\frac{|v|^2}{2}\right)\left(1+\log\left(1+\frac{|v|^2}{2}\right)\right)^{1-\alpha}}   \right|\right|\right|_{M^{3/2}(\IR^3_x\times\IR^3_v)}\\
&  \leq L^{2/3}\left|\left|\left| \frac{\p_t U(\tau,x)}{\left(1+\frac{|v|^2}{2}\right)\left(1+\log\left(1+\frac{|v|^2}{2}\right)\right)^{1-\alpha}   }   \right|\right|\right|_{L^{3/2}(\IR_v^3;M^{3/2}(\IR_x^3))} \\
&  \leq L^{2/3}|||\p_t U(\tau,\cdot)|||_{M^{3/2}(\IR_x^3)}\left(\int\limits_{\IR^3}    \frac{dv}{ \left(1+\frac{|v|^2}{2}\right)^{3/2} \left(1+\log\left(1+\frac{|v|^2}{2}\right)\right)^{3(1-\alpha)/2}}\right)^{2/3} \\
 & \leq c||\curr(\tau,\cdot)||_{L^1(\IR^3)},
\end{aligned}
\end{equation}
where the last integral is convergent since $3(1-\alpha)/2>1$.
{}From the inclusion $M^{3/2}(\IR^3_x\times\IR^3_v)\subset  L^1 (\IR^3_x\times\IR^3_v)+L^\infty (\IR^3_x\times\IR^3_v)$,
and integrating \eqref{iint2} over $B_r$, we get \eqref{claim} as desired.
\end{proof}

\section{Renormalized solutions and Lagrangian solutions}

We recall the different notions of weak solutions for the Vlasov-Poisson system.
We shall always assume that $1\leq N\leq 3$, and
we consider an initial datum $f^0 \in L^1(\irnxv)$, $f^0\geq 0$.
We introduce first renormalized solutions, following \cite{DPLVP,DPLkin}.
\begin{defin} We say that $f\in L^\infty((0,T);L^1(\irnxv))$, $f\geq 0$,
is a solution to the Vlasov equation \eqref{vlasov} in the renormalized sense if for all test functions $\beta\in C^1([0,\infty))$ with $\beta$ bounded, we have that
\begin{equation}
 \p_t \beta(f)+ v\cdot\nabla_x\beta(f)+ \dive_v\Bigl(E(t,x)\beta(f)\Bigr)=0,
\end{equation}
in $\mathcal{D}'((0,T)\times\IR^N_x\times\IR^N_v)$.
\end{defin}
We next introduce the notion of Lagrangian solutions.
\begin{defin}\label{lagrange}
Let be given a vector field $b(t,x,v)=(v,E(t,x))$ as in \eqref{bvf} for some $\rho\in L^\infty((0,T);L^1(\irn))$, $\rho\geq 0$,
and $\rho_b\in L^1(\IR^N)$. We assume that $E\in L^\infty((0,T);L^2(\irn))$, and that \eqref{eq:dtEdivj} holds
with $\curr\in  L^\infty((0,T);L^1(\irn))$.
We assume furthermore that either $N=1$ or $2$, or $N=3$ and $\omega=+1$, $\rho_b\in L^p(\IR^3)$ for some $p>3/2$.
We consider regular Lagrangian flows $Z$ as in Definition \ref{flow}, except that now $s\in[0,T]$
instead of $s\in[t,T]$ (forward-backward flow), and with compression constant $L$ independent of $t\in[0,T]$.
Then according to Subsections \ref{localintegr}, \ref{spatialreg}, Proposition \ref{estsuper2},
Proposition \ref{estsuper3}, the vector field $b$ satisfies assumptions {\bf(R1a), (R2a), (R3)}.
Therefore, Theorem \ref{Th estimatestab} yields the uniqueness of the forward-backward regular Lagrangian flow $Z=(X,V)$.
The whole theory of \cite{BC} then applies indeed, with very little modifications in the proofs.
In particular there is existence and uniqueness of the forward-backward regular Lagrangian flow, with compression constant $1$,
and stability. We can thus define in accordance with \cite{BC} a Lagrangian solution
$f$ to the Vlasov equation \eqref{vlasov} by
\begin{equation}
	f(t,x,v) = f^0\Bigl(X(s=0,t,x,v),V(s=0,t,x,v)\Bigr), \qquad \mbox{for all } t\in [0,T],
	\label{eq:vlasovlag}
\end{equation}
for arbitrary $f^0\in L^1(\irnxv)$.
It verifies in particular $f\in C([0,T];L^1(\irnxv))$, and it is indeed also a renormalized solution.
\end{defin}
\begin{defin}\label{lagrangeVP}
We define a Lagrangian solution to the Vlasov-Poisson system as a couple $(f,E)$ such that
\begin{enumerate}
\item  $f\in C([0,T];L^1(\irnxv))$, $f\geq 0$, $|v|^2f\in L^\infty((0,T);L^1(\irnxv))$, 
\item $E(t,x)$ is given by the convolution \eqref{field} with $\rho(t,x)=\int f(t,x,v) dv$,
$\rho_b\in L^1(\IR^N)$, $\rho_b\geq 0$ (and if $N=3$, $\omega=+1$, $\rho_b\in L^p(\IR^3)$ for some $p>3/2$),
\item $E\in L^\infty((0,T);L^2(\irn_x))$,
\item  The relation \eqref{eq:dtEdivj} holds with $\curr(t,x)=\int vf(t,x,v)dv$,
\item $f$ is a Lagrangian solution to the Vlasov equation, in the sense of \eqref{eq:vlasovlag}.
\end{enumerate}
\end{defin}

\section{Existence of Lagrangian solutions}

\subsection{Compactness}

In this subsection we prove two compactness results, Theorems \ref{cmpt} and \ref{cmptweak},
for families of Lagrangian solutions to the Vlasov-Poisson system,
with strongly or weakly convergent initial data.
\begin{lemma} \label{tauh}
Let $g(x)= \frac{x}{|x|^N}$ for $x\in\IR^N$, and denote by $\tau_h g(x)  = g(x+h)$.
Then for any $1<p<\frac{N}{N-1}$,
\begin{equation}
\|\tau_h g(x)-g(x)\|_{L^p(\irn)} \leq c |h|^\alpha,
\end{equation}
with $\alpha=1-N+N/p>0$, and where $c$ depends on $N,p$.
\end{lemma}
\begin{proof}
Fix $h\in\IR^N$, $h\not=0$. For $|x|>2|h|$, we have for all $0\leq\theta\leq 1$, $|x+\theta h|\geq |x|-\theta |h|>|x|/2$,
thus we have
\begin{equation}\begin{array}{l}
	\displaystyle|\tau_h g(x)-g(x)|\leq |h|\sup_{0\leq\theta\leq 1}|\nabla g(x+\theta h)|\\
	\displaystyle\hphantom{|\tau_h g(x)-g(x)|}
	\leq |h|\sup_{0\leq\theta\leq 1}\frac{c_N}{|x+\theta h|^N}\leq c_N \frac{|h|}{|x|^N}.
	\label{eq:estg}
	\end{array}
\end{equation}
Then we estimate
\begin{align*}
&  \int\limits_{ |x|>2|h|}    \frac{|h|^p}{|x|^{Np}}dx  = c_N|h|^p  \int\limits^\infty_{2|h|}r^{N-1-Np}dr
= c_N|h|^p \frac{(2|h|)^{N-Np}}{Np-N} = c_{N,p} |h|^{p-Np+N}.
\end{align*}
Next, for $|x|\leq 2|h|$, we write
\begin{equation}
 |\tau_h g(x)-g(x)|  \leq \frac{1}{|x+h|^{N-1}} + \frac{1}{|x|^{N-1}},
\end{equation}
and clearly
\begin{equation}\begin{array}{l}
	\displaystyle \int\limits_{|x|\leq 2|h|} \left(\frac{1}{|x+h|^{(N-1)p}}+\frac{1}{|x|^{(N-1)p}}\right)dx\\
	\displaystyle\leq 2\int\limits_{|y|\leq 3|h|} \frac{dy}{|y|^{(N-1)p}} =c_N\int\limits^{3|h|}_0  r^{N-Np+p-1} dr= c_{N,p}|h|^{N-Np+p},
	\label{eq:intxsmall}
	\end{array}
\end{equation}
since the last integral is convergent for $p<\frac{N}{N-1}$.
\end{proof}
\begin{theorem}\label{cmpt}
Let $(f_n,E_n)$ be a sequence of Lagrangian solutions to the Vlasov-Poisson system satisfying
\begin{equation}\label{convinit}
  f_n^0\to f^0 \text{ in } L^1(\irnxv),
\end{equation}
and
\begin{equation}
	\iint|v|^2f_n(t,x,v)dxdv+\int|E_n(t,x)|^2dx \leq C,\qquad
	\mbox{for all }t\in [0,T].
	\label{eq:boundenergy}
\end{equation}
Then, up to a subsequence $f_n$ converges strongly in $C([0,T];L^1(\irnxv))$ to $f$,
$E_n$ converges in $C([0,T];L^1_{loc}(\IR^N))$ to $E$, and $(f,E)$
is a Lagrangian solution to the Vlasov-Poisson system with initial datum $f^0$.
Moreover, the regular forward-backward Lagrangian flow $Z_n(s,t,x,v)$ converges to $Z(s,t,x,v)$
locally in measure in $\IR^N\times\IR^N$, uniformly in $s,t\in [0,T]$.
\end{theorem}
\begin{proof}
Step 1.) (Equi-integrability)
\\
Because of \eqref{eq:vlasovlag} and \eqref{convinit} we have
\begin{equation}
	\|f_n(t,\cdot,\cdot)\|_{L^1(\IR^N\times\IR^N)}= \|f_n^0\|_{L^1(\IR^N\times\IR^N)}\leq M.
	\label{eq:boundmass}
\end{equation}
Then because of the bounds \eqref{eq:boundmass}, \eqref{eq:boundenergy},
and applying Propositions \ref{estsuper2}, \ref{estsuper3}, one has for any $r>0$
\begin{equation}
	\La^{2N}\{(x,v)\in B_r: \sup_{0\leq s\leq T}|Z_n(s,t,x,v)|>\gamma\}\rightarrow 0,
	\mbox{ as }\gamma\rightarrow\infty,\mbox{uniformly in }t,n.
	\label{eq:unifsuper}
\end{equation}
Since the sequence $f_n^0$ is uniformly equi-integrable, and since $Z_n$ is measure-preserving,
we have by \eqref{eq:vlasovlag} and \eqref{eq:unifsuper} that $f_n(t,\cdot)$ is equi-integrable,
uniformly in $t,n$. Consequently, $\rho_n(t,\cdot)$ is also equi-integrable, uniformly in $t,n$.
Using the bound \eqref{eq:boundenergy}, $vf_n(t,\cdot)$ is also equi-integrable,
and therefore $\curr_n(t,\cdot)$ is also equi-integrable, uniformly in $t,n$.

Step 2.) (Compactness of the field)
\\
In order to prove that $E_n(t,x)\to E(t,x)\textrm{ in } L^1_{\loc}((0,T)\times\irn_x)$,
we first look at the compactness in $x$. 
Denote by $\tau_hE_n(t,x) = E_n(t, x+h)$. Then using \eqref{field},
\begin{equation}\begin{array}{l}
 \displaystyle\hphantom{\leq} ||\tau_h E_n(t,\cdot)-E_n(t,\cdot)||_{\lprn} \\
\displaystyle\leq c \left\|  \tau_h \frac{x}{|x|^N}-\frac{x}{|x|^N} \right\|_{\lprn}  ||\rho_n(t,\cdot)-\rho_b(\cdot)||_{\lrn}.
\end{array}
\end{equation}
Thus according to Lemma \ref{tauh}, we get for any $1<p<\frac{N}{N-1}$ that 
\begin{equation}
||E_n(t,x+h)-E_n(t,x)||_{\lprn}\to 0, \qquad\mbox{as } h\to 0,\mbox{ uniformly in }t,n.
\label{eq:compactE}
\end{equation}
Then, because of \eqref{eq:dtEdivj}, we have that $\partial_t E_n$ is bounded in
$L^\infty((0,T);\IS'(\irn))$. Applying Aubin's lemma, we conclude that
$E_n$ is compact in $L^1_{\loc}([0,T]\times\irn_x)$.
Thus after extraction of a subsequence, $E_n(t,x)\to E(t,x)$ strongly in $L^1((0,T);L^1_{\loc}(\irn_x))$.  

Step 3.) (Convergence of the flow)
\\
Because of the bound \eqref{eq:boundenergy}, one has $E\in L^\infty((0,T),L^2(\IR^N))$.
Also, using the uniform bounds on $\rho_n$, $\curr_n$ in $L^\infty((0,T);L^1(\IR^N))$ and
the uniform equi-integrability obtained in Step 1, one has up to a subsequence $\rho_n\rightarrow \rho$, $\curr_n\rightarrow \curr$
in the sense of distributions, with $\rho,\curr\in L^\infty((0,T);L^1(\IR^N))$.
We can pass to the limit in \eqref{field} and \eqref{eq:dtEdivj}.
Therefore, $b=(v,E)$ satisfies the assumptions {\bf(R1a), (R2a), (R3)} and Definition \ref{lagrange} applies.
According to \cite[Lemma 6.3]{BC}, since \eqref{eq:unifsuper} holds,
and $\rho_n$ are equi-integrable, we deduce the convergence of $Z_n$ to $Z$
locally in measure in $\IR^N_x\times\IR^N_v$, uniformly with respect to $s,t\in[0,T]$,
where $Z$ is the regular forward-backward Lagrangian flow associated to $b$.

Step 4.) (Convergence of $f$)
\\
Using the convergence \eqref{convinit}, we can apply \cite[Proposition 7.3]{BC},
and we conclude that $f_n\rightarrow f$ in $C([0,T];L^1(\IR^N_x\times\IR^N_v))$,
where $f$ is the Lagrangian solution to the Vlasov equation with coefficient $b$ and initial datum $f^0$.
It follows that $\rho_n\to \rho=\int fdv$ in $C([0,T];L^1(\irn_x))$.
By lower semi-continuity, we get from \eqref{eq:boundenergy} that
\begin{equation}
	\iint|v|^2f(t,x,v)dxdv+\int|E(t,x)|^2dx \leq C,\qquad
	\mbox{for all }t\in [0,T].
	\label{eq:boundenergyfinal}
\end{equation}
The bound \eqref{eq:boundenergy} gives also that $\curr_n\rightarrow J=\int vfdv$
in $C([0,T];L^1(\irn_x))$.
Therefore, $(f,E)$ is a Lagrangian solution to the Vlasov-Poisson system.
Using \eqref{field}, we get that $E_n\rightarrow E$ in $C([0,T];L^1_{loc}(\IR^N))$,
which concludes the proof.
\end{proof}
\begin{theorem}\label{cmptweak}
Let $(f_n,E_n)$ be a sequence of Lagrangian solutions to the Vlasov-Poisson system satisfying
\begin{equation}\label{weakconvinit}
  f_n^0\rightharpoonup f^0 \text{ weakly in } L^1(\irnxv),
\end{equation}
and the bound \eqref{eq:boundenergy}.
Then, up to a subsequence $f_n$ converges in $C([0,T];\mathrm{weak}-L^1(\irnxv))$ to $f$,
$E_n$ converges in $C([0,T];L^1_{loc}(\IR^N))$ to $E$, and $(f,E)$
is a Lagrangian solution to the Vlasov-Poisson system with initial datum $f^0$.
Moreover, the regular forward-backward Lagrangian flow $Z_n(s,t,x,v)$ converges to $Z(s,t,x,v)$
locally in measure in $\IR^N\times\IR^N$, uniformly in $s,t\in [0,T]$.
\end{theorem}
\begin{proof} It is the same as that of Theorem \ref{cmpt}, except the last step 4.
Instead we apply \cite[Proposition 7.7]{BC} and conclude that
$f_n\rightharpoonup f$ in $C([0,T];\mathrm{weak}-L^1(\IR^N_x\times\IR^N_v))$,
where $f$ is the Lagrangian solution to the Vlasov equation with coefficient $b$ and initial datum $f^0$.
It follows that $\rho_n\rightharpoonup \rho=\int fdv$ in $C([0,T];\mathrm{weak}-L^1(\irn_x))$.
By lower semi-continuity, we get again from \eqref{eq:boundenergy} the
energy bound \eqref{eq:boundenergyfinal}.
The bound \eqref{eq:boundenergy} also enables to conclude that
$\curr_n\rightharpoonup J=\int vfdv$ in $C([0,T];\mathrm{weak}-L^1(\irn_x))$.
Therefore, $(f,E)$ is a Lagrangian solution to the Vlasov-Poisson system.
Using \eqref{field} and the compactness estimate \eqref{eq:compactE},
we get that $E_n\rightarrow E$ in $C([0,T];L^1_{loc}(\IR^N))$, which concludes the proof.
\end{proof}

\subsection{Existence}

We conclude this section by the existence of Lagrangian solutions to the Vlasov-Poisson system
for initial datum in $L^1$ with finite energy, in the repulsive case.
\begin{theorem}\label{existence} Let $N=1,2$ or $3$, and let $f^0\in L^1(\IR^N_x\times\IR^N_v)$, $f^0\geq 0$.
Define $\rho^0$ and $E^0$ by
\begin{equation}
 \rho^0(x) = \int f^0(x,v) dv, \qquad E^0(x)=\frac{\omega}{|S^{N-1}|}\frac{x}{|x|^N}*(\rho^0(x)-\rho_b(x)),
\label{eq:initrhoE}
\end{equation}
with $\omega=+1$ (repulsive case),
$\rho_b\in L^1(\IR^N)$, $\rho_b\geq 0$, and in the case $N=3$ $\rho_b\in L^p(\IR^3)$ for some $p>3/2$.
Assume that the initial energy is finite,
\begin{equation}\label{energy0est}
 \iint|v|^2f^0(x,v)dxdv +\int|E^0(x)|^2dx <\infty.
\end{equation}
Then there exists a Lagrangian solution $(f,E)$ to the Vlasov-Poisson system
defined for all time, having $f^0$ as initial datum, and satisfying for all $t\geq 0$
\begin{equation}\label{prop}
 \iint|v|^2f(t,x,v)dxdv +\int|E(t,x)|^2dx   \leq\iint|v|^2f^0(x,v)dxdv +\int|E^0(x)|^2dx .
\end{equation}
\end{theorem}
\begin{proof}
We use the classical way of getting global weak solutions to the Vlasov-Poisson system,
i.e. we approximate the initial datum $f^0$ by a sequence of smooth data $f_n^0\geq 0$ with compact support.
We approximate also $\rho_b$ by smooth $\rho_b^n\geq 0$ with compact support (with $\int(\rho^0_n-\rho_b^n)dx=0$
if $N=1,2$). It is possible to do that with the upper bounds
\begin{equation}\begin{array}{l}
	\displaystyle \mathop{\mathrm{limsup}}_{n\to\infty}\iint|v|^2f^0_n(x,v)dxdv
	\leq \iint|v|^2f^0(x,v)dxdv,\\
	\displaystyle \mathop{\mathrm{limsup}}_{n\to\infty}\int|E^0_n(x)|^2dx
	\leq \int|E^0(x)|^2dx.
	\label{mollpres}
	\end{array}
\end{equation}
Then, for each $n$, there exists a smooth classical solution $(f_n,E_n)$ with initial datum $f^0_n$,
to the Vlasov-Poisson system, defined for all time $t\geq 0$. Note that we can alternatively consider
a regularized Vlasov-Poisson system.
Since $\omega=+1$, the conservation of energy \eqref{conservationenergy} gives for all $t\geq 0$,
\begin{equation} 
  \iint|v|^2f_n(t,x,v)dxdv +\int|E_n(t,x)|^2dx   =\iint|v|^2f_n^0 (x,v)dxdv +\int|E_n^0(x)|^2dx.
\end{equation}
The couple $(f_n,E_n)$ is in particular a Lagrangian solution to the Vlasov-Poisson system, for
all intervals $[0,T]$.
We can therefore apply Theorem \ref{cmpt}. Extracting a diagonal subsequence, we get the convergence
of $(f_n,E_n)$ to $(f,E)$ as stated in Theorem \ref{cmpt}, where $(f,E)$ is a Lagrangian
solution to the Vlasov-Poisson system defined for all time, with $f^0$ as initial datum.
The bound \eqref{eq:boundenergyfinal}, together with \eqref{mollpres}, gives \eqref{prop}.
\end{proof}
Let us end with a remark on measure densities.
When considering a sequence of solutions to the Vlasov-Poisson system, the vector fields
$b_n=(v,E_n)$ have a gradient in $(x,v)$ of the form
\begin{equation}\label{eq:r2bmat}
 Db_n  =
\left( \begin{matrix} 
   D_1 b_n^1 & D_2b_n^1\\
  D_1 b_n^2 & D_2b_n^2 
\end{matrix} \right)
=
\left( \begin{matrix} 
0 &     \mathrm{Id}\\
  S (\rho_n-\rho_b) &   0
\end{matrix} \right),
\end{equation}
where the index $1$ stands for $x$, $2$ for $v$, and where
$S$ is a singular integral operator.
If we require only that $D_1b_n^2$ converges in the sense of distributions to $D_1b^2=S(\rho-\rho_b)$,
for some measure $\rho\in \mathcal{M}(\irn)$, then we are in the setting of \cite{BBC}.
If $\rho_{n}$ is uniformly bounded in $L^1((0,T);\mathcal{M}(\IR^{N}))$,
and $b_n\to b$ strongly in $L^1 ((0,T);L^1_{\loc}(\irnxv))$ with $b$ satisfying \eqref{eq:r2bmat},
we conclude that $Z_n\to Z$ strongly, where $Z$ is the regular Lagrangian flow associated to $b$.
However, we are not able to define the push forward \eqref{eq:vlasovlag} of a measure $f^0$.
This prevents from applying the context of \cite{BBC} to the Vlasov-Poisson system with
measure data.

\end{document}